\documentclass[11pt]{article}
\usepackage[latin1]{inputenc}
\usepackage{amsmath}
\usepackage{amsfonts}
\usepackage{amssymb,amsthm, stmaryrd}
\usepackage{graphicx}
\usepackage{color}
\usepackage{multicol}
\usepackage{soul}
\usepackage[normalem]{ulem}
\usepackage[pagewise, displaymath, mathlines]{lineno}

\newcommand*\patchAmsMathEnvironmentForLineno[1]{%
  \expandafter\let\csname old#1\expandafter\endcsname\csname #1\endcsname
  \expandafter\let\csname oldend#1\expandafter\endcsname\csname end#1\endcsname
  \renewenvironment{#1}%
     {\linenomath\csname old#1\endcsname}%
     {\csname oldend#1\endcsname\endlinenomath}}%
\newcommand*\patchBothAmsMathEnvironmentsForLineno[1]{%
  \patchAmsMathEnvironmentForLineno{#1}%
  \patchAmsMathEnvironmentForLineno{#1*}}%
\AtBeginDocument{%
\patchBothAmsMathEnvironmentsForLineno{equation}%
\patchBothAmsMathEnvironmentsForLineno{align}%
\patchBothAmsMathEnvironmentsForLineno{flalign}%
\patchBothAmsMathEnvironmentsForLineno{alignat}%
\patchBothAmsMathEnvironmentsForLineno{gather}%
\patchBothAmsMathEnvironmentsForLineno{multline}%
}

\newtheorem{theorem}{Theorem}[section]

\newtheorem{problem}{Problem}
\newtheorem{corollary}[theorem]{Corollary}
\newtheorem{lemma}[theorem]{Lemma}
\newtheorem{remark}[theorem]{Remark}

\newcommand{\Z}{\mbox{${\mathbb Z}$}}

\title{Zero-sum $K_m$ over $\Z$ and the story of $K_4$}

\begin{document}

\maketitle

\begin{center}

\begin{multicols}{2}

Yair Caro\\[1ex]
{\small Dept. of Mathematics and Physics\\
University of Haifa-Oranim\\
Tivon 36006, Israel\\
yacaro@kvgeva.org.il}

\columnbreak

Adriana Hansberg\\[1ex]
{\small Instituto de Matem\'aticas\\
UNAM Juriquilla\\
Quer\'etaro, Mexico\\
ahansberg@im.unam.mx}\\[2ex]

\end{multicols}

Amanda Montejano\\[1ex]
{\small UMDI, Facultad de Ciencias\\
UNAM Juriquilla\\
Quer\'etaro, Mexico\\
amandamontejano@ciencias.unam.mx}\\[4ex]

\end{center}

\begin{abstract}
We prove the following results solving  a problem raised in Caro-Yuster \cite{CY3}.
For a positive integer $m\geq 2$, $m\neq 4$, there are infinitely many values of $n$ such that the following  holds: There is a weighting function  $f:E(K_n)\to \{-1,1\}$ (and hence a  weighting function  $f: E(K_n)\to \{-1,0,1\}$), such that $\sum_{e\in E(K_n)}f(e)=0$ but, for every copy $H$ of $K_m$ in $K_n$, $\sum_{e\in E(H)}f(e)\neq 0$. On the other hand, for every integer $n\geq 5$  and every weighting function $f:E(K_n)\to \{-1,1\}$ such that $|\sum_{e\in E(K_n)}f(e)|\leq \binom{n}{2}-h(n)$, where $h(n)=2(n+1)$ if $n \equiv 0$ (mod $4$) and $h(n)=2n$ if $n \not\equiv 0$ (mod $4$), there is always a copy $H$ of $K_4$ in $K_n$ for which $\sum_{e\in E(H)}f(e)=0$, and the value of $h(n)$ is sharp.  

 
 
\end{abstract}

\section{Introduction}

Our main source of motivation is a recent paper of Caro and Yuster \cite{CY3}, extending classical zero-sum Ramsey theory to weighting functions $f:E(K_n)\to \{-r,-r+1, \cdots ,0, \cdots , r-1,r\}$  seeking zero-sum copies of a given graph $H$ subject to the obviously necessary condition that $|\sum_{e\in E(K_n)}f(e)|$ is bounded away from $ r\binom{n}{2}$, or  even in the extreme case where $|\sum_{e\in E(K_n)}f(e)|=0$.
 
In zero-sum Ramsey theory, one studies functions $f:E(K_n)\to X$, where $X$ is usually the cyclic group $\Z_k$ or (less often) an arbitrary finite abelian group. The goal is to show that, under some necessary divisibility conditions imposed on the number of the edges $e(H)$ of a graph $H$ and for sufficiently  large $n$, there is always a zero-sum copy of $H$, where by a zero-sum copy of $H$ we mean a copy of $H$ in $K_n$ for which $\sum_{e\in E(H)}f(e)=0$ (where $0$  is the neutral element of $X$).  For several results concerning zero-sum Ramsey theory for graphs see \cite{AC,BD1,BD2,C1,C2,CY1,FK,SS},
for zero-sum Ramsey problems concerning matrices and  linear algebraic techniques see \cite{BCRY,CY2,WW,W}.
 
The following  notation was introduced in \cite{CY3} and the following  zero-sum problems over $\Z$ 
were considered. For positive integers $r$ and $q$, an \emph{$(r,q)$-weighting} of the edges of the complete graph $K_n$ is a function $f:E(K_n)\to \{-r, \cdots ,r\}$  such that  $|\sum_{e\in E(K_n)}f(e)|<q$. The general problem considered in \cite{CY3} is to find nontrivial conditions on the $(r,q)$-weightings that guarantee the existence of certain bounded-weight subgraphs and even zero-weighted subgraphs (also called \emph{zero-sum} subgraphs). So, given a subgraph $H$ of $K_n$, and a weighting  $f:E(K_n)\to \{-r,\cdots ,r\}$, the \emph{weight} of $H$ is defined as $w(H)=\sum_{e\in E(H)}f(e)$, and we say that $H$ is a \emph{zero-sum graph} if $w(H)=0$. Finally, we say that a weighting function $f:E(K_n) \to \{-1,1\}$ is \emph{zero-sum-$H$ free} if it contains no zero-sum copy of $H$.

Among the many results proved in \cite{CY3}, the following theorem and open problem are the main motivation of this paper.
 
\begin{theorem}[Caro and Yuster, \cite{CY3}]\label{thm:CY}
For a real $\epsilon >0$ the following holds. For $n$ sufficiently large, any weighting $f:E(K_n)\to \{-1,0,1\}$ with $|\sum_{e\in E(K_n)}f(e)|\leq (1-\epsilon)n^2/6$ contains a zero-sum copy of $K_4$.
On the other hand, for any positive integer $m$ which is not of the form $m = 4d^2$, there are infinitely many integers $n$ for which there is a weighting $f:E(K_n)\to \{-1,0,1\}$ with $\sum_{e\in E(K_n)}f(e)=0$ without a zero-sum copy of $K_m$.  
\end{theorem} 

The authors posed the following complementary  problem:

\begin{problem}[Caro and Yuster, \cite{CY3}]\label{probl}
For an integer $m = 4d^2$, is it true that, for $n$ sufficiently large, any weighting $f:E(K_n)\to \{-1,0,1\}$  with $\sum_{e\in E(K_n)}f(e)=0$ contains a zero-sum copy of $K_m$?
\end{problem}

The main result in this paper is a negative answer to the above problem for any $m \geq 2$ except for $m=4$ already for weightings of the form $f:E(K_n)\to \{-1,1\}$ with $\sum_{e\in E(K_n)}f(e)=0$. On the other hand, concerning the study of the existence of zero-sum copies of $K_4$,  we prove a result analogous to Theorem \ref{thm:CY} where the range $\{-1,1\}$ instead of $\{-1,0,1\}$ is considered. Finally, we show that Theorem \ref{thm:CY} can neither be extended to wider ranges. To be more precise, we gather our results in the following theorem.

\begin{theorem}\label{thm:main} \hspace{1cm}
\begin{enumerate}

\item For any positive integer $m \geq2$, $m\neq 4$, there are infinitely many values of $n$ such that the following  holds: There is a weighting function $f:E(K_n)\to \{-1,1\}$ with $\sum_{e\in E(K_n)}f(e)=0$ which is zero-sum-$K_m$ free.  

\item Let $n$ be an integer such that $n\geq 5$. Define $g(n)=2(n+1)$ if $n \equiv 0$ (mod $4$) and $g(n)=2n$ if $n \not\equiv 0$ (mod $4$). Then, for any weighting $f:E(K_n)\to \{-1,1\}$ such that $|\sum_{e\in E(K_n)}f(e)|\leq \binom{n}{2}-g(n)$, there is a zero-sum copy of $K_4$. 

\item There are infinitely many values of $n$ such that the following holds:
There is a weighting function $f:E(K_n)\to \{-2,-1,0,1,2\}$ with $\sum_{e\in E(K_n)}f(e)=0$  which is zero-sum-$K_4$ free.

\end{enumerate}
\end{theorem}

Theorem \ref{thm:main} together with the above Theorem \ref{thm:CY} do not only solve Problem \ref{probl}, they also supply a good understanding of the situation concerning $K_4$ as the value of $g(n)$ is sharp and the upper bound $(1-\epsilon)n^2/6$ in Theorem \ref{thm:CY} is nearly sharp, as already observed in \cite{CY3}.
 
We will use the following notation. Given a  weighting $f:E(K_n)\to \{-r, \cdots, r\}$ and $i\in \ \{-r, \cdots, r\}$, denote by $E(i)$ the set of  the $i$-weighted edges, that is, $E(i)=f^{-1}(i)$  and define $e(i)=|E(i)|$.  Given a vertex $x\in V(K_n)$ we use $deg_{i}(x)$ to denote the number of $i$-weighted edges incident to $x$, that is, $deg_{i}(x)=|\{u:f(xu)=i\}|$.

In Section \ref{sec:K4} we will prove instances 2 and 3 of Theorem \ref{thm:main}, corresponding to the study of the existence of zero-sum copies of $K_4$. In order to prove instance 2, we will use an equivalent formulation consequence of the following remark.

\begin{remark}\label{rem:eq}
A weighting $f:E(K_n)\to \{-1,1\}$ satisfies $\left|\sum_{e\in E(K_n)}f(e)\right|\leq \binom{n}{2}-g(n)$ if and only if $\min\{e(-1),e(1)\}\geq \frac{1}{2}g(n)$.
\end{remark}
The remark follows from the fact that $e(1)+e(-1)=\binom{n}{2}$, which implies $|\sum_{e\in E(K_n)}f(e)|=|e(1)-e(-1)|=\max\{e(-1),e(1)\}-\min\{e(-1),e(1)\}=\binom{n}{2}-2\min\{e(-1),e(1)\}$.

In Section \ref{sec:K4} we will also prove that instance 2 of Theorem \ref{thm:main} is best posible by exhibiting, for each $n\geq5$, a weighting function $f:E(K_n)\to \{-1,1\}$ with $\min\{e(-1),e(1)\}=   \frac{1}{2}g(n)-1$ and no zero-sum copies of $K_4$. Moreover, we will characterize the extremal functions.

Finally, relying heavily on Pell equations and some classical biquadratic Diophantine equations,
 in Section \ref{sec:Kk} we will prove instance 1 of Theorem \ref{thm:main}, corresponding to the study of the existence of zero-sum copies of $K_m$ in $0$-weighted weightings, where $m \neq 4$.
  

\section{The case of $K_4$}\label{sec:K4}

We will use standard graph theoretical notation to denote particular graphs. Having said this, $K_{1,3}$ will stand for the star with three leaves, $K_3 \cup K_1$ for the disjoint union of a triangle and a vertex, $P_k$ for a path with $k$ edges, and $C_k$ for a cycle with $k$ edges.

A weighting function $f:E(K_n) \to \{-1,1\}$ is \emph{zero-sum-$K_4$ free} if and only if the graph induced by $E(-1)$ (or equivalently $E(1)$) is $\{K_{1,3}, K_3 \cup K_1, P_3\}$-free (in the induced sense). The following lemma, which characterizes the $K_3$-free subclass of the family of $\{K_{1,3}, K_3 \cup K_1, P_3\}$-free graphs, will be useful in proving the forthcoming results. 
We define
\[h(n)= \left\{ \begin{array}{rl}
                  n+1, & \mbox{ if } n\equiv 0  \mbox{ (mod $4$), and} \\     
                 n,  & \mbox{ otherwise. }  \\              
\end{array}\right.\]

\begin{lemma}\label{lem:triangle-free}
Let $G$ be a $\{K_{1,3}, K_3, P_3\}$-free graph on $n$ vertices. Then each component of $G$ is isomorphic to one of $C_4$, $K_1$, $K_2$ or $P_2$. Moreover, $e(G) \le h(n)-1$, and equality holds if and only if $G \cong J \cup \bigcup_{i=1}^{q} C_4$, where $J \in \{\emptyset, K_1, K_2, P_2\}$ and $q = \lfloor \frac{n}{4} \rfloor$.
\end{lemma}

\begin{proof}
Let $J$ be a connected component of $G$. If $J$ has at most $3$ vertices, then it is easy to see that $J \in \{K_1, K_2, P_2\}$. So assume that $J$ has at least $4$ vertices. Then, since $J$ is $\{K_3, P_3\}$-free, we infer that $J$ has no vertex of degree larger than $2$ and so we can deduce that $J \cong C_4$. Further, we note that $e(J) = |J|$ if $J \cong C_4$, and $e(J) = |J|-1$ otherwise. This implies that, among all $\{K_{1,3}, K_3, P_3\}$-free graphs on $n$ vertices, $G$ has maximum number of edges if and only if $G \cong J \cup \bigcup_{i=1}^{q} C_4$, where $J \in \{\emptyset, K_1, K_2, P_2\}$ and $q = \lfloor \frac{n}{4} \rfloor$. Since, clearly, $e(J \cup \bigcup_{i=1}^{q} C_4) = h(n) -1$, the proof is complete.
\end{proof}

\begin{lemma}\label{lem:one_is_K3-free}
Let $n\geq 5$ and $f:E(K_n)\to \{-1,1\}$ be a zero-sum-$K_4$ free coloring. Let $G_{-1}$ and $G_1$ be the graphs induced by $E(-1)$ and $E(1)$, respectively.  Then at least one of $G_{-1}$ or $G_1$ is triangle-free.
\end{lemma}

\begin{proof}
Suppose for contradiction that both $G_{-1}$ and $G_1$ have a triangle. Let $abc$ be a triangle in $G_{-1}$ and $uvw$ a triangle in $G_1$. Suppose first that $abc$ and $uvw$ have a vertex in common, say $a=u$. Consider the graph $J$ induced by the $(-1)$-edges among vertices in $\{a,b,c,v,w\}$. If a vertex $x \in \{b,c\}$ is neighbor of both $v$ and $w$, then $\{x,v,w,a\}$ would induce a $K_{1,3}$ in $G_{-1}$, which is not possible. If no vertex $x \in \{b,c\}$ is adjacent to some $y \in \{v,w\}$, then $\{a,b,c,y\}$ would induce a $K_3\cup K_1$ in $G_{-1}$, which again is not possible. Hence, $\{b,c,v,w\}$ induces two independent edges. But then $\{b,c,v,w\}$ induces a $P_3$ in $G_{-1}$, a contradiction. Hence, we can assume that any pair of triangles such that one has only $(-1)$-edges and the other only $1$-edges are vertex disjoint. This implies that from any vertex in $\{u,v,w\}$ there is at most one $(-1)$-edge to vertices from $\{a,b,c\}$. Analogously, from any vertex in $\{a,b,c\}$ there is at most one $1$-edge to vertices from $\{u,v,w\}$. But this implies that there are at most $6$ edges between $\{a,b,c\}$ and $\{u,v,w\}$, which is false.
Since in all cases we obtain a contradiction, we conclude that at least one of $G_{-1}$ or $G_1$ is triangle-free.
\end{proof}

By Remark \ref{rem:eq}, the next result is equivalent to instance 2 of Theorem \ref{thm:main}.

\begin{theorem}\label{thm:k4}
Let $n\geq 5$ and $f:E(K_n)\to \{-1,1\}$ such that $\min\{e(-1),e(1)\}\geq h(n)$. Then  there is a zero-sum $K_4$. 
\end{theorem}

\begin{proof}
Let $f:E(K_n)\to \{-1,1\}$ be such that $\min\{e(-1),e(1)\}\geq h(n)$ and suppose for contradiction that it has no zero-sum $K_4$. Let $G_{-1}$ and $G_1$ be the graphs induced by $E(-1)$ and $E(1)$, respectively. Then both $G_{-1}$ and $G_1$ are $\{K_{1,3}, K_3 \cup K_1, P_3\}$-free graphs. By Lemma \ref{lem:one_is_K3-free}, $G_{-1}$ or $G_1$ is $K_3$-free. So we may assume, without loss of generality, that $G_{-1}$ is triangle-free. It follows by Lemma \ref{lem:triangle-free} that $e(-1) = |E(G_{-1})| \le h(n)-1$, which is a contradiction to the hypothesis.
\end{proof}

The following theorem shows that Theorem \ref{thm:k4} is best possible and characterizes the extremal zero-sum-$K_4$ free weightings. We will use Mantel's Theorem, that any graph on $n$ vertices and at least $\frac{n^2}{4}+1$ edges contains a copy of $K_3$.

\begin{theorem}\label{thm:k4_sharp}
Let $n\geq 5$ and $f:E(K_n)\to \{-1,1\}$ such that $e(1) = h(n)-1$. Then $f$ is zero-sum-$K_4$ free if and only if the graph induced by $E(1)$ is isomorphic to $J \cup \bigcup_{i=1}^{q} C_4$, where $J \in \{\emptyset, K_1, K_2, P_2\}$ and $q = \lfloor \frac{n}{4} \rfloor$.
\end{theorem}

\begin{proof}
If the graph induced by $E(1)$ is isomorphic to $J \cup \bigcup_{i=1}^{q} C_4$, where $J \in \{\emptyset, K_1, K_2, P_2\}$ and $q = \lfloor \frac{n}{4} \rfloor$, it is easy to check that $f$ is zero-sum-$K_4$ free. Conversely, let $f$ be zero-sum-$K_4$ free. Then the graphs $G_{-1}$ and $G_1$ induced by $E(-1)$ and $E(1)$, respectively, are both $\{K_{1,3}, K_3 \cup K_1, P_3\}$-free. If $n = 5$, it is easy to check that the only $\{K_{1,3}, K_3 \cup K_1, P_3\}$-free graph with $h(5) - 1 = 4$ edges is isomorphic to $C_4 \cup K_1$, and so we are done. Hence, we may assume that $n \ge 6$. Observe that 
\[e(-1) = \frac{n(n-1)}{2} - h(n)+1 \ge \frac{n(n-1)}{2} - n = \frac{n(n-3)}{2},\]
whose right-hand side is at least $\frac{n^2}{4}$ for $n \ge 6$. Hence, by Mantel's Theorem, $G_{-1}$ has a triangle, and, by Lemma \ref{lem:one_is_K3-free}, this implies that $G_1$ is triangle-free. It follows that $G_1$ is a $\{K_{1,3}, K_3, P_3\}$-free graph on $n$ vertices and with $h(n) - 1$ edges. Thus, with Lemma \ref{lem:triangle-free}, we obtain that $G_1$ is isomorphic to $J \cup \bigcup_{i=1}^{q} C_4$, where $J \in \{\emptyset, K_1, K_2, P_2\}$ and $q = \lfloor \frac{n}{4} \rfloor$, and we are done.
\end{proof}

The following theorem is instance 2 from Theorem \ref{thm:main}. It shows that, whenever we take a wider range for the weighting function $f$, we cannot hope for a result as in Theorem \ref{thm:k4} anymore.

\begin{theorem}\label{thm:larger_range}
There are infinitely many values of $n$ such that the following holds:
There is a weighting function $f:E(K_n)\to \{-2,-1,0,1,2\}$ with $\sum_{e\in E(K_n)}f(e)=0$ which is zero-sum-$K_4$ free.
\end{theorem}

\begin{proof}
Let $X \cup Y$ be a partition of the vertex set of $K_n$ and consider the weighting function $f:E(K_n)\to \{-2,-1,0,1,2\}$ such that 
\[
f(uv) = \left\{\begin{array}{ll}
-2, & \mbox{if } u,v \in X\\
1, & \mbox{if } u,v \in Y\\
0, & \mbox{otherwise}.
\end{array}
\right.
\]
Clearly, $f$ is zero-sum-$K_4$ free. On the other hand, $\sum_{e\in E(K_n)}f(e)=0$ if and only if
\[
-2\frac{|X|(|X|-1)}{2} + \frac{|Y|(|Y|-1)}{2} = 0,
\]
which is equivalent to $(2|Y|-1)^2-2(2|X|-1)^2 = -1$. Hence, solving the latter equation is equivalent to solve the following Pell's equation
\begin{equation}\label{eq:pell-pythago}
y^2-2x^2 = -1,
\end{equation}
for (odd) integers $x = 2|X|-1$ and $y= 2|Y|-1$. It is well-known that the Diophantine equation $y^2-2x^2=\pm 1$ has infinitely many solutions given by
\[x_{k}=\frac{a^k-b^k}{a-b}=\frac{a^k-b^k}{2\sqrt{2}}, \hspace{2ex} y_k=\frac{a^k+b^k}{2},\]
where $a=1+\sqrt{2}$, $b=1-\sqrt{2}$ and $k \in \mathbb{N}$. Moreover, since  $y_k^2 - 2x_k^2 = (-1)^k$, the solutions for equation (\ref{eq:pell-pythago}) are the pairs $(x_k,y_k)$ where $k$ is odd. Observe also that, for odd $k$, $x_k$ and $y_k$ are odd, too. Hence, each odd $k$ gives us a solution $(\frac{x_k+1}{2}, \frac{y_k+1}{2})$ for $(|X|,|Y|)$ and thus for $n = \frac{x_k+y_k}{2}+1$ and we are done.
\end{proof}

For the sake of comprehension, let us compute small values of  $n=\frac{x_k+y_k}{2}+1$ and exhibit how the partition  $(|X|,|Y|)=(\frac{x_k+1}{2}, \frac{y_k+1}{2})$ gives a zero-sum weighting function $f$ as described in the theorem. Recall that we only want to consider solutions for  odd values of $k$. 
So we have $(x_1,x_3,x_5,\dots)=(1,5,29,\dots)$ and  $(y_1,y_3,y_5,\dots)=(1,7,41,\dots)$, and the corresponding sequence of $n$'s is $(2,7,36,\dots )$. The case of $n=2$ is not interesting for vaquity reasons. For $n=7$, the partition is  $(|X|,|Y|)=(3,4)$, thus there will be $ \binom{3}{2}$ edges weighted with $-2$,   $ \binom{4}{2}$ edges weighted with $1$ and the rest of edges weighted with $0$, adding up to zero. For $n=36$,  the partition is  $(|X|,|Y|)=(15,21)$, and the sum of weighted edges is $-2\binom{15}{2}+1\binom{21}{2}=(-2) \cdot 105+1 \cdot 210=0$.

\section{The case of $K_m$, $m\neq 4$}\label{sec:Kk}

A \emph{balanced} $\{-1,1\}$-weighting function $f:E(K_n)\to \{-1,1\}$ is a function for which $e(-1)=e(1)$. 
In Section \ref{sec:K4}, we prove that, for $n\geq 5$, any function $f:E(K_n)\to \{-1,1\}$ with sufficiently many edges assigned to each type contains a zero-sum $K_4$. In this section, we prove that this is not true  for $K_m$ with $m \in \mathbb{N} \setminus \{1, 4\}$. In other words, we exhibit, for infinitely many values of $n$, the existence of a balanced weighting function $f:E(K_n)\to \{-1,1\}$ without a zero-sum copies of $K_m$, where $m \neq 1,4$. In order to define those functions, consider first the following Pell equation:
\begin{equation}\label{eq:pell}
8x^2-8x+1=y^2.
\end{equation}
It is well known that such a Diophantine equation has infinitely many solutions given by the recursion 
\[(x_1,y_1)=(1,1),\]
\[(x_2,y_2)=(3,7),\]
\[y_k=6y_{k-1}-y_{k-2}, \hspace{.5cm} x_{k}=\frac{y_k+x_{k-1}+1}{3}.\]


\begin{lemma}\label{lem:bal1}
Let $n$ be a positive integer and consider the complete graph $K_n$ and a partition $V(K_n) = A \cup B$ of its vertex set. Then the function $f:E(K_n)\to \{-1,1\}$ defined as
\[f(e)= \left\{ \begin{array}{rl}
                   -1, & \mbox{ if } e\subset A \\     
                  1,  & \mbox{ otherwise, }  \\              
\end{array}\right.\]
is balanced if and only if $n = \frac{1+y_k}{2}$ and $|A| = x_k$ for some $k \in\mathbb{N}$.
\end{lemma}

\begin{proof}
Suppose first that $f$ is balanced and let $|A| = x$. Then 
$$e(-1)=\frac{x(x-1)}{2} = \frac{1}{2} \binom{n}{2},$$ which yields
\[
n^2-n-(2x^2-2x)=0,
\]
and therefore $n = \frac{1+\sqrt{8x^2-8x+1}}{2}$. But this is only an integer if $8x^2-8x+1 = y^2$ for some integer $y$, and we obtain equation (\ref{eq:pell}). Hence, $|A| = x = x_k$ and $n = \frac{1+y_k}{2}$ for some $k \in \mathbb{N}$.\\
Conversely, suppose that $n = \frac{1+y_k}{2}$ and $|A| = x_k$ for some $k \in\mathbb{N}$. Then
\[
n = \frac{1+y_k}{2} = \frac{1+\sqrt{y_k^2}}{2} = \frac{1+\sqrt{8x_k^2-8x_k+1}}{2}.
\]
Thus $n$ is the positive root of
\begin{equation}\label{eq:n_k}
n^2-n-(2x_k^2-2x_k)=0,
\end{equation}
which is equivalent to
\[
\frac{x_k(x_k-1)}{2}=\frac{1}{2}\binom{n}{2}.
\]
Since the left hand side of this equation is precisely $e(-1)$ and the right hand side is half the number of the edges of $K_n$, it follows that $f$ is balanced.
\end{proof}

\begin{lemma}\label{lem:bal2}
Let $n$ be a positive integer and consider the complete graph $K_n$ and a partition $V(K_n) = A \cup B$ of its vertex set. Then the function \[f(e)= \left\{ \begin{array}{rl}
                   -1, & \mbox{ if } e\subset A  \mbox{ or } e\subset B \\     
                  1,  & \mbox{ otherwise, }  \\              
\end{array}\right.\]
is balanced if and only if $n = k^2$ and $\{|A|, |B| \} = \{\frac{1}{2}k(k+1), \frac{1}{2}k(k-1)\}$ for some $k \in\mathbb{N}$.
\end{lemma}

\begin{proof}
Suppose first that $f$ is balanced and let $|A| = w$. Then 
\[
e(1) = w(n-w)=\frac{1}{2} \binom{n}{2},
\]
which is equivalent to
\[
w^2 - nw + \frac{1}{4}n(n-1)=0.
\]
Hence, 
\begin{equation}\label{eq:n(w)}
w = \frac{n \pm \sqrt{n}}{2},
\end{equation}
which is an integer if and only if $n = k^2$ for some $k \in \mathbb{N}$. So we obtain $n = k^2$ and $w \in \{\frac{1}{2}k(k+1), \frac{1}{2}k(k-1)\}$. Since $|B| = n - |A| = k^2 - w$, it follows easily that $\{|A|, |B| \} = \{\frac{1}{2}k(k+1), \frac{1}{2}k(k-1)\}$ and we are done.\\
Conversely, suppose that $n = k^2$ and $\{|A|, |B| \} = \{\frac{1}{2}k(k+1), \frac{1}{2}k(k-1)\}$ for some $k \in\mathbb{N}$. Without loss of generality, assume that $|A| = \frac{1}{2}k(k+1)$. Then
\[
e(1) = |A| (n - |A|) = \frac{1}{2}k(k+1) \left( k^2 - \frac{1}{2}k(k+1)\right) = 
\frac{1}{4} k^2(k^2-1) = \frac{1}{2} \binom{n}{2},\]
implying that $f$ is balanced.
\end{proof}

We define the set $S_1$ as the set of all integers $n_k = \frac{1+y_k}{2}$, $k \in \mathbb{N}$, where $(x_k,y_k)$ is the k-th solution of (\ref{eq:pell}), that is,
$$S_1 = \left\{\frac{1+y_k}{2} \; | \; k \in \mathbb{N} \right\}.$$
Further, let $S_2$ be the set of all integer squares, that is,
$$S_2 = \left\{k^2 \; | \; k \in \mathbb{N} \right\}.$$

Lemmas \ref{lem:bal1} and \ref{lem:bal2} yield the following corollary.

\begin{corollary}\label{cor:Km_Si}
For any integer $m \in \mathbb{N} \setminus (S_1 \cap S_2)$, there are infinitely many positive integers $n$ such that there exists a balanced function $f:E(K_n)\to \{-1,1\}$  without zero-sum $K_m$.
\end{corollary}

\begin{proof}
Let $m \in \mathbb{N} \setminus (S_1 \cap S_2)$. By Lemmas \ref{lem:bal1} and \ref{lem:bal2}, there is a balanced function $f:E(K_n)\to \{-1,1\}$ for each $n \in S_1 \cup S_2$. Suppose that there is a zero-sum $K_m$ in such a weighting $f$ for a given $n \in S_1 \cup S_2$. Then, the function $f$ restricted to the edges of $K_m$ is a balanced function on $E(K_m)$, which is not possible by Lemmas \ref{lem:bal1} and \ref{lem:bal2} since $m \in \mathbb{N} \setminus S_1 \cap S_2$. Since $S_1 \cup S_2$ has infinitely many elements, it follows that there are infinitely many positive integers $n$ such that there exists a balanced function $f:E(K_n)\to \{-1,1\}$  without zero-sum $K_m$.
\end{proof}

Now we can state the main result of this section, which is equivalent to instance 3 of Theorem \ref{thm:main}.

\begin{theorem}
For any integer $m \in \mathbb{N} \setminus \{1, 4\}$, there are infinitely many positive integers $n$ such that there exists a balanced weighting  $f:E(K_n)\to \{-1,1\}$   which is zero-sum-$K_m$ free.
\end{theorem}

\begin{proof}
By Corollary \ref{cor:Km_Si}, for  any $m \in \mathbb{N} \setminus (S_1\cap S_2)$, there are infinitely many positive integers $n$, such that there exists a balanced weighting function $f:E(K_n)\to \{-1,1\}$ without a zero-sum $K_m$. We will show that $S_1 \cap S_2=\{1,4\}$. Let $q$ be an integer such that $q^2\in S_1$ (and thus $q^2\in S_1\cap S_2$). Then $q^2$ must be the positive root of equation (\ref{eq:n_k}) for some $x_k$. Thus we need to know for which positive integers $q$ and $x$ the following is possible:
\begin{equation}\label{eq:qx}
q^4-q^2-(2x^2-2x)=0.
\end{equation}
Note that equation (\ref{eq:qx}) can be written as:
\begin{equation}\label{eq:QX}
Q^2-2X^2=-1.
\end{equation}
where $Q=2q^2-1$ and $X=2x-1$.  Again (as in the proof of Theorem \ref{thm:larger_range}), we have to deal with the Diophantine equation $Q^2-2X^2=\pm 1$, which has infinitely many solutions given by
\[Q_k=\frac{a^k+b^k}{2}, \hspace{.5cm} X_{k}=\frac{a^k-b^k}{a-b}=\frac{a^k-b^k}{2\sqrt{2}},\]
where $a=1+\sqrt{2}$ and $b=1-\sqrt{2}$. Since we need to solve equation (\ref{eq:QX}) (that is, with $-1$ on the right side), we know that $k$ must be odd. Therefore, according to the definition of $Q$, we need to determine all odd $k$'s such that $$Q_k+1=2q^2,$$ or equivalently, $$2Q_k+2=4q^2,$$ and so,
\begin{equation}\label{eq:ab}
a^k+b^k+a+b=(2q)^2.
\end{equation}
We consider two cases:\\

\noindent
\emph{Case 1.} If $k\equiv 1$ (mod $4$),  we will prove that the left side of equation (\ref{eq:ab}) is $4Q_{\frac{k-1}{2}}Q_{\frac{k+1}{2}}$. Note that $ab=-1$ and, since in this case $\frac{k-1}{2}$ is even, we have $(ab)^{\frac{k-1}{2}}=(-1)^{\frac{k-1}{2}}=1$. Hence,
\begin{align*}
a^k+b^k+a+b&=a^k+b^k+(ab)^{\frac{k-1}{2}}(a+b)\\
 &=a^{\frac{k-1}{2}}a^{\frac{k+1}{2}}+b^{\frac{k-1}{2}}b^{\frac{k+1}{2}}+a^{\frac{k-1}{2}}b^{\frac{k-1}{2}}(a+b)\\
 &=a^{\frac{k-1}{2}}a^{\frac{k+1}{2}}+b^{\frac{k-1}{2}}b^{\frac{k+1}{2}}+a^{\frac{k+1}{2}}b^{\frac{k-1}{2}}+a^{\frac{k-1}{2}}b^{\frac{k+1}{2}}\\
 &=(a^{\frac{k-1}{2}}+b^{\frac{k-1}{2}})(a^{\frac{k+1}{2}}+b^{\frac{k+1}{2}})\\
 &=4Q_{\frac{k-1}{2}}Q_{\frac{k+1}{2}}.
\end{align*}
Thus, by (\ref{eq:ab}), we conclude that $Q_{\frac{k-1}{2}}Q_{\frac{k+1}{2}}$ is a perfect square. We know that, for all $i$, $Q_i$ and $Q_{i+1}$ are coprimes. Thus, it follows that both $Q_{\frac{k-1}{2}}$ and $Q_{\frac{k+1}{2}}$ are perfect squares. Coming back to equation (\ref{eq:QX}), the following must be satisfied 
\begin{equation}\label{eq:YX}
Y^4-2X^2=-1
\end{equation}
where $Q_{\frac{k+1}{2}}=Y^2$. But, the only possible solution for the Diophantine equation (\ref{eq:YX}) is $(Y,X)=(\pm 1,1)$. Hence, $Q_{\frac{k+1}{2}}= 1$, which means that $k=1$, and so $Q = Q_1 = 1$. Since $Q=2q^2-1$ and $q>0$, we conclude that $q=1$.\\

\noindent
\emph{Case 2.} If $k\equiv 3$ (mod $4$), then we will prove that the left side of equation (\ref{eq:ab}) is $8X_{\frac{k-1}{2}}X_{\frac{k+1}{2}}$. Recall that $ab=-1$ and, since in this case $\frac{k+1}{2}$ is even, we have $(ab)^{\frac{k+1}{2}}=(-1)^{\frac{k+1}{2}}=1$. Hence,
\begin{align*}
a^k+b^k+a+b&=a^k+b^k+(ab)^{\frac{k+1}{2}}(a+b)\\
&=a^k+b^k-(ab)^{\frac{k-1}{2}}(a+b)\\
 &=a^{\frac{k-1}{2}}a^{\frac{k+1}{2}}+b^{\frac{k-1}{2}}b^{\frac{k+1}{2}}-a^{\frac{k-1}{2}}b^{\frac{k-1}{2}}(a+b)\\
 &=a^{\frac{k-1}{2}}a^{\frac{k+1}{2}}+b^{\frac{k-1}{2}}b^{\frac{k+1}{2}}-a^{\frac{k+1}{2}}b^{\frac{k-1}{2}}-a^{\frac{k-1}{2}}b^{\frac{k+1}{2}}\\
 &=(a^{\frac{k-1}{2}}-b^{\frac{k-1}{2}})(a^{\frac{k+1}{2}}-b^{\frac{k+1}{2}})\\
 &=8X_{\frac{k-1}{2}}X_{\frac{k+1}{2}}.
\end{align*}
Thus, by (\ref{eq:ab}), we conclude that $2X_{\frac{k-1}{2}}X_{\frac{k+1}{2}}$ is a perfect square. We know that $X_{\frac{k-1}{2}}$and $X_{\frac{k+1}{2}}$ have different parity.  Observe that, for $k\equiv 3$ (mod $4$), $X_{\frac{k-1}{2}}$ is odd and $X_{\frac{k+1}{2}}$ is even. Since for all $i$, $X_i$ and $X_{i+1}$ are coprimes, also   $X_{\frac{k-1}{2}}$ and $2X_{\frac{k+1}{2}}$ are coprimes, from which it follows that both $X_{\frac{k-1}{2}}$ and $2X_{\frac{k+1}{2}}$ are perfect squares. Particularly, coming back to equation (\ref{eq:QX}), we obtain
\begin{equation}\label{eq:QW}
Q^2-2W^4=-1
\end{equation}
where $X_{\frac{k-1}{2}}=W^2$. Note that equation (\ref{eq:QW}) is the well known Ljunggren Equation $1+Q^2=2W^4$. Such a Diophantine equation has solutions only for $W=1$ and $W=13$, which correspond respectively to $X_1$ and $X_7$ (because $X_1=1=1^2$ and $X_7=169=13^2$). Therefore, we have two possibilities, either $k=3$ (that is $X_{\frac{3-1}{2}}=X_1$), or $k=15$ (that is $X_{\frac{15-1}{2}}=X_7$). The second case is disclaimed since  $X_{\frac{15+1}{2}}=X_8=408=2 \cdot (204)$ and $204$ is not a perfect square. The first case, corresponding to $k=3$, leads to $X_{\frac{3-1}{2}}=X_1=1$ and $X_{\frac{3+1}{2}}=X_2=2$. The solution $(Q_1,X_1)=(1,1)$ gives $q=1$ as we saw in Case 1. The solution $(Q_2,X_2)=(3,2)$ gives  $q=2$ (since $Q=2q^2-1$ an $q>0$). 

From both cases we conclude that, if $q^2\in S_1\cap S_2$ then either $q=1$ or $q=2$. Hence, $S_1\cap S_2=\{1,4\}$ which concludes the proof.
\end{proof}

\section{Conclusions}

While the situation about zero-sum copies of $K_m$ over $\Z$-weightings is fairly clear now, 
a lot of interesting results can be proved when the graphs in question are not complete graphs. 
Several examples are given in \cite{CY3} (for example, certain complete bipartite graphs and many more), and in a forthcoming paper \cite{CHM} which is under preparation.

\section{Acknowledgements}

We would like to thank our colleague Florian Lucca for some fruitful discussions concerning some results of this paper.


The second author was partially supported by PAPIIT IA103217 and CONACyT project 219775. The third author was partially supported by PAPIIT IN114016 and CONACyT project 219827. Finally, we would like to acknowledge the support from Center of Innovation in Mathematics, CINNMA A.C.


\begin{thebibliography}{10}

\bibitem{AC} N. Alon , Y. Caro.  On three zero-sum Ramsey-type problems,
\emph{Journal of Graph Theory}, {\bf 17} (1993), 177--192.
 
\bibitem{BCRY}  P. Balister, Y. Caro, C. Rousseau, and R. Yuster. Zero-sum square matrices,
\emph{European Journal of Combinatorics}, {\bf 23} (2002), 489--497.
 
\bibitem{BD1}  A. Bialostocki and P. Dierker. Zero sum Ramsey theorems,
\emph{Congressus Numerantium}, {\bf 70} (1990), 119--130.
 
\bibitem{BD2} A. Bialostocki and P. Dierker. On zero sum Ramsey numbers: multiple copies of a graph.
\emph{Journal of Graph Theory}, {\bf 18} (1994), 143--151.
 
\bibitem{C1} Y. Caro. A complete characterization of the zero-sum (mod 2) Ramsey numbers,
\emph{Journal of Combinatorial Theory, Series A}, {\bf 68} (1994), 205--211 .
 
\bibitem{C2}  Y. Caro. Zero-sum problems - a survey,
\emph{Discrete Mathematics}, {\bf 152} (1996),  93--113.


\bibitem{CHM} Y. Caro, A. Hansberg, A. Montejano, Zero-sum graphs over $\Z$-weightings of $K_n$, under preparation. 
 
 
\bibitem{CY1} Y. Caro and R. Yuster. The characterization of zero-sum (mod 2) bipartite Ramsey numbers,
\emph{Journal of Graph Theory}, {\bf 429} (1998),  151--166.
 
\bibitem{CY2}  Y. Caro and R. Yuster. The uniformity space of hypergraphs and its applications,
\emph{Discrete Mathematics}, {\bf 202} (1999), 1--19.
 
\bibitem{CY3}  Y. Caro and R. Yuster. On zero-sum and almost zero-sum subgraphs over $\Z$,
 \emph{Graphs and Combinatorics}, {\bf 32} (2016), 49--63.
 
\bibitem{FK}  Z. Furedi and D. Kleitman. On zero-trees.
\emph{Journal of Graph Theory}, {\bf 16} (1992), 107--120.
 
\bibitem{SS}  A. Schrijver and P.D. Seymour. A simpler proof and a generalization of the zero-trees theorem,
\emph{Journal of Combinatorial Theory, Series A}, {\bf 58} (1991), 301--305.
 
\bibitem{WW}  R. M. Wilson and  T. W. H. Wong.   Diagonal forms of incidence matrices associated with $t$-uniform hypergraphs, \emph{European Journal of Combinatorics}, {\bf 35} (2014), 490--508.  
 
\bibitem{W}  T. W. H. Wong.  Diagonal forms and zero-sum (mod 2) bipartite Ramsey numbers,
\emph{Journal of Combinatorial Theory, Series A}, {\bf 124} (2014), 97--113
\end{thebibliography}
\end{document}